\documentclass{amsart}
\usepackage{amstext,amssymb,amsthm,amsopn,newlfont,graphpap,graphics,graphicx,mathrsfs,enumitem}
\allowdisplaybreaks
\usepackage[parfill]{parskip}
\usepackage[noadjust]{cite}
\usepackage{epigraph}
\usepackage[colorlinks=true,
            linkcolor=red,
            urlcolor=blue,
            citecolor=magenta]{hyperref}
\usepackage{color}
\usepackage{mathrsfs}
\allowdisplaybreaks
\theoremstyle{plain}
\newtheorem{thm}{Theorem}[section]
\newtheorem*{thm*}{Theorem}

\newtheorem*{prop*}{Proposition}

\newtheorem*{cor*}{Corollary}

\newtheorem*{lem*}{Lemma}
\theoremstyle{definition}

\newtheorem*{defn*}{Definition}

\newtheorem*{exmps*}{Examples}

\newtheorem*{exmp*}{Example}

\newtheorem*{exerc*}{Exercise}

\newtheorem*{rems*}{Remarks}

\newtheorem*{rem*}{Remark}
\newcommand{\N}{{\mathbb N}}

\DeclareMathOperator{\spa}{span}
\begin{document}
\title[On a characterization of finite-dimensional vector spaces]
{On a characterization\\ of finite-dimensional vector spaces}
\author[Marat V. Markin]{Marat V. Markin}
\address{
Department of Mathematics\newline
California State University, Fresno\newline
5245 N. Backer Avenue, M/S PB 108\newline
Fresno, CA 93740-8001, USA
}
\email{mmarkin@csufresno.edu}
\begin{abstract}
We provide a characterization of the finite dimensionality of vector spaces in terms of the right-sided invertibility of linear operators on them.
\end{abstract}
\maketitle

\section{Introduction}

In paper \cite{Markin2005}, found is a characterization of one-dimensional (real or complex) normed algebras in terms 
of the bounded linear operators on them, echoing the celebrated \textit{Gelfand-Mazur theorem} charachterizing complex one-dimensional Banach algebras (see, e.g., \cite{Bach-Nar,Gelfand39,Gelfand41,Naimark,Rickart1958}).

Here, continuing along this path, we provide a simple characterization of the finite dimensionality of vector spaces in terms of the right-sided invertibility of linear operators on them.

\section{Preliminaries}

As is well-known (see, e.g., \cite{Horn-Johnson,ONan}), a square matrix $A$ with complex entries is invertible \textit{iff} it is one-sided invertible, i.e., there exists
a square matrix $C$ of the same order as $A$ such that
\[
AC = I\ \text{(\textit{right inverse})}\quad \text{or}\quad CA=I\ \text{(\textit{left inverse})},
\]
where $I$ is the \textit{identity matrix} of an appropriate size, in which case $C$ is the (two-sided)  inverse of $A$, i.e.,
\[
AC=CA=I.
\]

Generally, for a linear operator on a (real or complex) vector space, the existence of a \textit{left inverse} implies being \textit{invertible}, i.e., \textit{injective}. Indeed, let $A:X\to X$ be a linear operator on a (real or complex) vector space $X$ and a linear operator $C:X\to X$ be its \textit{left inverse}, i.e.,
\begin{equation}\label{cfdvs2}
CA=I,
\end{equation}
where $I$ is the \textit{identity operator} on $X$. Equality \eqref{cfdvs2}, obviously, implies that 
\[
\ker A=\left\{0\right\},
\]
and hence, there exists an inverse $A^{-1}:R(A)\to X$ for the operator $A$, where $R(A)$ is its range (see, e.g., \cite{Markin2020EOT}). Equality \eqref{cfdvs2} also implies that the inverse operator $A^{-1}$ is the restriction of $C$ to $R(A)$.

Further, as is easily seen, for a linear operator on a (real or complex) vector space, the existence of a \textit{right inverse}, i.e., a linear operator $C:X\to X$ such that
\begin{equation*}
AC=I,
\end{equation*}
immediately implies being \textit{surjective}, which, provided the underlying vector space is \textit{finite-dimensional}, by the \textit{rank-nullity theorem} (see, e.g., \cite{Markin2018EFA,Markin2020EOT}), is equivalent to being \textit{injective}, i.e., being \textit{invertible}.

With the underlying space being \textit{infinite-dimensional}, the arithmetic of infinite cardinals does not allow to directly infer by the \textit{rank-nullity theorem} that the \textit{surjectivity} of a linear operator on the space is equivalent to its \textit{injectivity}. In this case the right-sided invertibility for linear operators need not imply invertibility. For instance, on the (real or complex) \textit{infinite-dimensional} vector space $l_\infty$ of bounded sequences, the  left shift linear operator
\[
l_\infty\ni x:=(x_1,x_2,x_3,\dots)
\mapsto Lx:=(x_2,x_3,x_4,\dots)\in l_\infty
\]
is \textit{non-invertible} since
\[
\ker L=\left\{(x_1,0,0,\dots) \right\}\neq \left\{0\right\}
\]
(see, e.g., \cite{Markin2018EFA,Markin2020EOT}), but the right shift linear operator
\[
l_\infty\ni x:=(x_1,x_2,x_3,\dots)
\mapsto Rx:=(0,x_1,x_2,\dots)\in l_\infty
\]
is its \textit{right inverse}, i.e.,
\[
LR=I,
\]
where $I$ is the \textit{identity operator} on $l_\infty$.

Not only does the above example give rise to the natural question of whether, when the right-sided invertibility for linear operators on a (real or complex) vector space implies their invertibility, i.e., \textit{injectivity}, the underlying space is necessarily \textit{finite-dimensional} but also serve as an inspiration for proving the \textit{``if''} part of the subsequent characterization.

\section{Characterization}

\begin{thm}[Characterization of Finite-Dimensional Vector Spaces]\ \\
A (real or complex) vector space $X$ is finite-dimensional
iff, for linear operators on $X$, right-sided invertibility 
implies invertibility.
\end{thm}

\begin{proof}\

\textit{``Only if''} part. Suppose that the vector space $X$ is \textit{finite-dimensional} with $\dim X=n$ ($n\in \N$) and let $B:=\left\{x_1,\dots,x_n\right\}$ be an ordered basis for $X$.

For an arbitrary linear operator $A:X\to X$ on $X$, which has
a \textit{right inverse}, i.e., a linear operator $C:X\to X$ such that
\begin{equation*}
AC=I,
\end{equation*}
where $I$ is the \textit{identity operator} on $X$, let
$[A]_B$ and $[C]_B$ be the \textit{matrix representations} of
the operators $A$ and $C$ relative to the basis $B$, respectively (see, e.g., \cite{Horn-Johnson,ONan}).

Then
\begin{equation}\label{cfdvs1}
[A]_B[C]_B=I_n,
\end{equation}
where $I_n$ is the \textit{identity matrix} of size $n$
(see, e.g., \cite{Horn-Johnson,ONan}).

By the \textit{multiplicativity of determinant} (see, e.g, \cite{Horn-Johnson,ONan}), equality \eqref{cfdvs1}
implies that
\[
\det\left([A]_B\right)\det\left([C]_B\right)=\det\left([A]_B[C]_B\right)=\det(I_n)=1.
\]

Whence, we conclude that
\[
\det\left([A]_B\right)\neq 0,
\]
which, by the \textit{determinant characterization of invertibility}, implies that the matrix $[A]_B$ is invertible, and hence, so is the operator $A$ (see, e.g., \cite{Horn-Johnson,ONan}).

\textit{``If''} part. Let us prove this part \textit{by contrapositive}, assuming that the vector space $X$ is \textit{infinite-dimensional}. Suppose that $B:=\left\{x_i\right\}_{i\in I}$ is a
(Hamel) basis for $X$ (see, e.g., \cite{Markin2018EFA,Markin2020EOT}), where $I$ is an infinite indexing set, and that
$J:=\left\{i(n)\right\}_{n\in \N}$ is a \textit{countably infinite} subset of $I$. 

Let us define a linear operator $A:X\to X$ as follows:
\[
Ax_{i(1)}:=0,\ Ax_{i(n)}:=x_{i(n-1)},\ n\ge 2,\quad Ax_i:=x_i,\ i\in I\setminus J,
\]
and
\[
X\ni x=\sum_{i\in I}c_ix_i\mapsto Ax:=\sum_{i\in I}c_iAx_i,
\]
where 
\[
\sum_{i\in I}c_ix_i
\]
is the \textit{basis representation} of a vector $x\in X$ relative to $B$, in which all but a finite number of the coefficients $c_i$, $i\in I$, called the \textit{coordinates} of $x$ relative to $B$, are zero (see, e.g., \cite{Markin2018EFA,Markin2020EOT}).

As is easily seen, $A$ is a linear operator on $X$, which is \textit{non-invertible}, i.e., \textit{non-injective}, since
\[
\ker A=\spa\left(\left\{x_{i(1)}\right\}\right)\neq \left\{0\right\}.
\]

The linear operator $C:X\to X$ on $X$ defined as follows:
\[
Cx_{i(n)}:=x_{i(n+1)},\ n\in\N,\quad Cx_i:=x_i,\ i\in I\setminus J,
\]
and
\[
X\ni x=\sum_{i\in I}c_ix_i\mapsto Cx:=\sum_{i\in I}c_iCx_i,
\]
is a \textit{right inverse} for $A$ since
\[
ACx_{i(n)}=Ax_{i(n+1)}=x_{i(n)},\ n\in \N,\quad ACx_i=Ax_i=x_i,\ i\in I\setminus J.
\]

Thus, on a (real or complex) infinite-dimensional vector space, there exists a non-invertible linear operator with a right inverse, which completes the proof of the \textit{``if''} part, and hence, of the entire statement.
\end{proof}  


\end{document}